\documentclass[11pt]{amsart}

\usepackage{float}
\usepackage{tikz}
\usepackage{amsmath, amsthm, color, fullpage}
\usepackage{graphicx, amsaddr}
\usepackage{enumerate}

\newtheorem{theorem}{Theorem}[section]

\newtheorem{conjecture}[theorem]{Conjecture}
\newtheorem{lemma}[theorem]{Lemma}

\title{1-planar graphs are odd 13-colorable }
\author{Runrun Liu$^{1}$, Weifan Wang$^{1}$,      Gexin Yu$^{2}$}

\address{
$^{1}$\small School of Mathematics, Zhejiang Normal University, Jinhua, Zhejiang 321004, China.\\
$^{2}$\small Department of Mathematics, William \& Mary, Williamsburg, VA 23185, USA.
}

\thanks{The research of the first author was supported by NSFC, China (No. 12101563) and Natural Science Foundation of Zhejiang Province (ZJNSFC), China (No. LQ22A010011). The research of the second author was supported by NSFC, China (No. 11771402;12031018)}

\email{rliu1206@zjnu.edu.cn,wwf@zjnu.edu.cn,gyu@wm.edu}

\begin{document}
\maketitle
\begin{abstract}
An odd coloring of a graph $G$ is a proper coloring such that any non-isolated vertex in $G$ has a coloring appears odd times on its neighbors.  The odd chromatic number, denoted by $\chi_o(G)$, is the minimum number of colors that admits an odd coloring of $G$.  Petru\v{s}evski and \v{S}krekovski in 2021 introduced this notion and proved that if $G$ is planar, then $\chi_o(G)\le9$ and conjectured that $\chi_o(G)\le5$. More recently, Petr and Portier improved $9$ to $8$. A graph is $1$-planar if it can be drawn in the plane so that each edge is crossed by at most one other edge. Cranston, Lafferty and Song showed that every $1$-planar graph is odd $23$-colorable. In this paper, we improved this result and showed that every $1$-planar graph is odd $13$-colorable.
\medskip


\end{abstract}
 \baselineskip=16pt

\section{Introduction}

A proper k-coloring of $G$ is a mapping $f:V(G)\to[k]$ such that $f(u)\ne f(v)$ whenever $uv\in E(G)$, where $[k]=\{1,2,\cdots,k\}$. An {\em odd coloring} of a graph $G$ is a proper coloring such that any non-isolated vertex in $G$ has a color appearing odd times on its neighbors. The {\em odd chromatic number}, denoted by $\chi_o(G)$, is the minimum number of colors required in odd colorings of $G$.

The study of odd coloring of graphs was recently initiated by Petru\v{s}evski and \v{S}krekovski~\cite{PS21+},
as a relaxation of conflict-free coloring in the literature, which was first studied for hypergraphs.  A
coloring of vertices of a hypergraph is {\em conflict-free} if at least one vertex in each (hyper-)edge has
a unique color, see~\cite{ELRS03}. Cheilaris~\cite{C09} studied the conflict-free coloring of graphs. More specifically,
a graph G has a conflict-free coloring if for any vertex $v\in V(G)$, there is a color occurs exactly once
in $N(v)$. Conflict-free colorings have been widely studied, see~\cite{GST14,KKL12,PT09}. The odd coloring
of hypergraphs was introduced by Cheilaris et al. ~\cite{CKP13}, which requires each edge in $E(H)$ has some
color that occurs odd number of times on its vertices. This problem has been studied in~\cite{BMWW08,FG16}.

Clearly, $\chi(G)\le\chi_o(G)$, and furthermore,  $\chi_o(G)-\chi(G)$ can be arbitrarily large. In fact, consider the graph $G$ obtained by subdividing each edge once in a complete graph $K_n$. We can observe that $\chi(G)=2$ but $\chi_o(G)\ge n$. Note that a proper subgraph of $G$ may have higher odd chromatic number than that of $G$. For example, $C_4$ is a subgraph of $K_4-e$, but we have $\chi_o(C_4)=4>3=\chi_o(K_4-e)$.
The Four Color Theorem states that every planar graph has chromatic number $4$. Oberserve that $\chi_o(C_5)=5$ for a $5$-cycle $C_5$. Petru\v{s}evski and \v{S}krekovski~\cite{PS21+} proposed the following conjecture.

\begin{conjecture}
For every planar graph $G$ it holds that $\chi_o(G)\le5$.
\end{conjecture}

In the same paper, they proved that $\chi_o(G)\le9$ for each planar graph $G$. Petr and Portier~\cite{PP22+} improved $9$ to $8$. For planar graph $G$ with girth at least $g$, Cranston~\cite{C22+} showed that $\chi_o(G)\le 6$ if $g=6$, and $\chi_o(G)\le 5$ if $g=7$; Cho, Choi, Kwon and Park~\cite{CCKP22} showed that $\chi_o(G)\le 7$ if $g=5$,  $\chi_o(G)\le 4$ if $g=11$, and $\chi_o(G)\le c$ if $g=\lceil\frac{4c}{c-2}\rceil$, where $c\ge 5$.
Let $mad(G)=\max_{H\subseteq G}2|E(H)|/|V(H)|$ be the maximum average degree of a graph $G$. Cranston~\cite{C22+} and Cho-Choi-Kwon-Park~\cite{CCKP22} also studied $\chi_o(G)$ when $mad(G)$ is bounded. In particular, it is shown in \cite{CCKP22} that $\chi_o(G)\le c$ if $c\ge 7$ and $mad(G)<\frac{4c}{c+2}$, and  $\chi_o(G)\le 4$ if $mad(G)\le \frac{22}{9}$ and $G$ contains no induced $5$-cycle.

A graph is {\em 1-planar} if it can be drawn in the plane so that each edge is crossed by at most one other edge. The notion of 1-planarity was initially introduced by Ringel~\cite{R65} in 1965, who proved that $\chi(G)\le7$ for any 1-planar graph $G$ and conjectured that every 1-planar graph is 6-colorable due to the 1-planar graph $K_6$. This conjecture was solved by Borodin~\cite{B84}, who also gave a new proof~\cite{B95} later.

A number of interesting results about structures and parameters of 1-planar graphs has been obtained in recent years. Fabrici and Madaras~\cite{FM07} proved that every 1-planar graph $G$ has $|E(G)|\le4|V(G)|-8$, implying $\delta(G)\le7$, and constructed a 1-planar graph $G$ with $\delta(G)=7$. Czap and Hud\'{a}k~\cite{CH13}  showed that every 1-planar graph can be decomposed into a planar graph and a forest. Dvo\v{r}\'{a}k, Lidick\'{y} and Mohar~\cite{DLM17} showed that a 1-planar graph drawn in the plane so that the distance between every pair of crossing is at least 15 is 5-choosable. 
A proper vertex coloring of a graph is acyclic if every cycle uses at least three
colors. Borodin et al.~\cite{BKRS01} proved that every 1-planar graph is acyclically 20-colorable, and this result was improved to 18 by Yang, Wang and Wang~\cite{YWW20}. For more results about 1-planar graphs, we refer readers to a survey by Kobourov, Liotta and Montecchiani~\cite{KLM17}. 

In this paper, we focus on the odd coloring of 1-planar graphs. Cranston, Lafferty and Song~\cite{CLS22+} showed that every 1-planar graph is odd 23-colorable. We improved this result to 13.

\begin{theorem}\label{main}
Let $G$ be a 1-planar graph, then $\chi_o(G)\le13$.
\end{theorem}

Let $K_7^*$ be the graph obtained from the complete graph $K_7$ by subdividing each dege exactly once. Note that $K_7^*$ is a $1$-planar graph and $\chi_o(K_7^*)=7$. This is best lower bound we know about odd chromatic number of 1-planar graphs.

We introduce some notations used in this paper.  A vertex of degree $k$, at least $k$, and at most $k$
is called a {\em $k$-vertex}, a {\em $k^+$-vertex}, and {\em a $k^-$-vertex}, respectively.
Usually, we use $[u_1u_2\cdots u_k]$ to denote a face $f$ if $u_1,u_2,\ldots,u_k$ are the boundary vertices of
$f$ in a cyclic order. We say that a $k$-face $f=[u_1u_2\cdots u_k]$ is an $(a_1,a_2\ldots,a_k)$-{\em face}
 if the degree of the vertex $u_i$ is $a_i$ for $i =1,2,\ldots,k$.
A face of degree $k$, at least $k$, and at most $k$
is called a {\em $k$-face}, a {\em $k^+$-face}, and {\em a $k^-$-face}, respectively. An {\em odd color} of a vertex $v$ is a color that appears an odd number of times on $N(v)$.  If $\phi$ is an odd coloring of $G$ and $v\in V (G)$, we use $\phi_o(v)$ to denote an odd color of $v$ (arbitrarily choose one if there are more than one odd colors).

\section{Structures on minimal counterexamples to Theorem~\ref{main} }

In this section, we assume that $G$ is a 1-planar graph that does not have an odd coloring of $13$ colors, but any graph with fewer vertices than $G$ has one. Consider $G^*$ is a plane drawing of a 1-planar graph such that each edge has at most one crossing point, and all crossing points are as few as possible. So, for each pair of crossing edges $x_1y_1, x_2y_2\in E(G)$, the four end vertices $x_1,y_1,x_2,y_2$ are pairwise distinct.

Let $X(G)$ denote the set of crossing points in $G$. The associated plane graph, denoted $G^*$, of $G$ is a plane graph with $V(G^*)=V(G)\cup X(G), E(G^*)=E_0(G)\cup E_1(G)$, where $E_0(G)$ stands for the set of non-crossed edges in $G$ and $E_1(G)=\{xz,zy|xy\in E(G)\setminus E_0(G) \text{ and } z \text{ is a crossing point on } xy\}$. It is easy to observe that $d_{G^*}(v)=d_G(v)$ for each vertex $v\in V(G)$, and $d_{G^*}(v)=4$ for each vertex $v\in X(G)$, which we shall denote as a $4^*$-vertex. Note that no two vertices in $X(G)$ are adjacent by the 1-planarity. In a proper coloring of $G$, a vertex of odd degree will always has an odd color. The following Lemmas~\ref{lem1} to~\ref{lem6} are from ~\cite{CLS22+}. We write them in a little more general setting in the appendix, whose proofs are the same as those in~\cite{CLS22+}.

\begin{lemma}\label{lem1} (Claim 1 in ~\cite{CLS22+})
The graph $G$ is $2$-edge-connected and thus $\delta(G)\ge 2$.
\end{lemma}

\begin{lemma}\label{lem2} (Claim 2 in ~\cite{CLS22+})
Every vertex of odd degree in $G$ has degree at least $7$.
\end{lemma}

\begin{lemma}\label{lem3} (Claim 4 in ~\cite{CLS22+})
Every edge incident to a $6^-$-vertex in $G$ has a crossing.
\end{lemma}

\begin{lemma}\label{lem5} (Claim 5 in ~\cite{CLS22+})
The graph $G^*$ has no loop or $2$-face. Every $3$-face in $G^*$ is incident to either three $7^+$-vertices or two $7^+$-vertices and one $4^*$-vertex.
\end{lemma}

\begin{lemma}\label{lem6} (Claim 6 in ~\cite{CLS22+})
Every $2$-vertex $v$ in $G^*$ is incident to a $5^+$-face and a $4^+$-face.
\end{lemma}

For convenience, we call a vertex $v$ {\em easy} in $G$ if $v$ satisfies one of the following conditions:  (i) $d(v)\le6$; (ii) $d(v)\ge7$ and $d(v)$ is odd; (iii) $v$ has a neighbor of degree at most $6$. And we have the following observation.
\

{\bf Observation:} Let $u\in V(G)$. Let $\phi$ be an odd $c$-coloring of $G-u$. If $v$ is an easy neighbor of $u$, then we only need to forbid at most one color in $\{\phi(v),\phi_o(v)\}$ when coloring $u$. In fact,
\begin{itemize}
\item if $d(v)\le6$, then we may uncolor $v$ and therefore $u$ is forbidden to use one color $\phi_o(v)$. Note that after $G-v$ is colored, we can color $v$ by forbidding at most $12$ colors.

\item if $d(v)$ is odd, then $u$ is forbidden to use one color $\phi(v)$ since $v$ will always has an odd color.

\item if $v$ has a neighbor $v'$ of degree at most $6$, then we may uncolor $v'$ and therefore $u$ is forbidden to use one color $\phi(v)$. Note that after $G-v'$ is colored, we can color $v'$ by forbidding at most $12$ colors.
\end{itemize}

Our key contribution is the following lemma, which is a strengthening of Lemma 2.1 in \cite{CCKP22}.

\begin{lemma}\label{lem4}
If $v$ is of odd degree or has a $6^-$-neighbor, then $d(v)\ge7$ and has at most $(2d(v)-13)$ easy neighbors.
\end{lemma}

\begin{proof}
First we show that $d(v)\ge7$. Suppose otherwise that $d(v)\le6$. By our assumption, $G-v$ has an odd $13$-coloring $\phi$. Then $v$ can be colored by avoiding at most $12$ colors, namely, at most $6$ colors used for $v$'s neighbors and at most $6$ additional colors in regard to oddness concerning the neighbors of $v$. If $v$ is of odd degree, there is at least one color appearing an odd number of times on $N_G(v)$. If $v$ has a $6^-$-neighbor $v'$, then we can recolor $v'$ by forbidding at most $12$ colors when $v$ has no odd color. In either case, we get an odd $13$-coloring of $G$ that extends $\phi$, a contradiction.

Now we show that $v$ has at most $(2d(v)-13)$ easy neighbors. Note it is trivial for $d(v)\ge13$. So assume that $d(v)\le12$. Suppose otherwise that $v$  has at least $(2d(v)-12)$ easy neighbors, thus has at most $d(v)-(2d(v)-12)=12-d(v)$ other neighbors. By our assumption, $G-v$ has an odd $13$-coloring $\phi$. Then $v$ can be colored by avoiding at most $(2d(v)-12)+2(12-d(v))=12$ colors, namely, at most one color in $\{\phi(w),\phi_o(w)\}$ for each easy neighbor $w$ of $v$ by the observation and at most two colors $\phi(w),\phi_o(w)$ for each non-easy neighbor $w$ of $v$. If $v$ is of odd degree, there is at least one color appearing an odd number of times on $N(v)$. If $v$ has a $6^-$-neighbor $v'$, then we can recolor $v'$ by forbidding at most $12$ colors when $v$ has no odd color. In either case, we get an odd $13$-coloring of $G$ that extends $\phi$, a contradiction.
\end{proof}

\begin{lemma}\label{lem7}
There is no $4$-face in $G^*$ with two $2$-vertices and two $4^*$-vertices.
\end{lemma}

\begin{proof}
Suppose otherwise. Then in $G$, there are $2$-vertices $u,v$ with $N_G(u)=\{u_1,u_2\}$ and $N_G(v)=\{v_1,v_2\}$ such that $uu_1$ intersects $vv_1$ and $uu_2$ intersects $vv_2$. Then we can modify the embedding to avoid the crossings by switching the locations of $u$ and $v$, a contradiction.
\end{proof}

\begin{lemma}\label{lem8}
The graph $G^*$ contains no $(7,7,8)$-faces.
\end{lemma}

\begin{proof}
Suppose otherwise that $f=[uvw]$ is a $(7,7,8)$-face in $G^*$. By our assumption, $G-\{u,v\}$ has an odd $13$-coloring $\phi$. Then we can first color $v$ by forbidding at most 12 colors. Now we need to forbid at most $13$ colors when color $u$, namely, at most $7$ colors on $N_G(u)$ and at most $6$ colors in regard to oddness of its neighbors. Since $u$ cannot be colored, we may assume that $\phi_o(w)$ is the unique odd color of $w$, which implies that there are at most four colors on $N(w)-u$. Then we uncolor $w$ and color $u$ with $\phi(w)$. Then $w$ can be colored by forbidding at most $11$ colors, namely, at most $5$ colors on $N_G(w)$ and at most $6$ colors in regard to oddness of its neighbors. Since $\phi(w)$ on $u$ is unique in $N_G(w)$, we get an odd $13$-coloring of $G$, a contradiction.
\end{proof}

\section{Proof of Theorem~\ref{main}}

In this section, we will give a proof of Theorem~\ref{main} by way of discharging method on $G^*$.

For each vertex $x\in V(G^*)$ or face $x\in F(G^*)$, let its initial charge be $\mu(x)=d(x)-4$. Then by Euler formula, $\sum_{x\in V(G^*)\cup F(G^*)}=-8$. We will design a set of rules to redistribute the charges among vertices and faces of $G^*$, so that each $x\in V(G^*)\cup F(G^*)$ has a final charge $\mu^*(x)\ge 0$. This would give us a contradiction since $-8=\sum_x \mu(x)=\sum_x\mu^*(x)\ge 0$.

Note that only $2$-vertices and $3$-faces in $G^*$ has negative initial charges.

In the rest of this paper, for any vertex $v$, let  $d(v)=k$ and $N_G^*(v)=\{v_i:1\le i\le k\}$. Let $f_i$ be the face containing $v_i,v,v_{i+1}$ (index module $k$) for each $i\in[k]$ in $G^*$. We call a $2$-vertex $v$ {\em special} if $v$ is on a $4$-face in $G^*$, and a $7$-vertex $v$ {\em special} if $v$ is incident to seven $3$-faces such that $f_7$ is a $(7,7,10^+)$-face with $d(v_7)=7,d(v_1)\ge10$ and $f_i$ is a $(4^*,7,7^+)$-face for each $i\in[6]$ with $4^*$-vertices $v_2,v_4,v_6$. We call a face $f$ {\em poor} if it is a $k$-face with $k\in\{3,4,6\}$ satisfies one of the following condition:(i) $f$ is $(7,7,10^+)$-face with two special $7$-vertices; (ii) $f$ is a $4$-face $(u, 4^*, v, 4^*)$ such that $x,y$ are easy neighbors of $u$ in $G$ and $ux,uy$ intersect the two edges incident with the $2$-vertex $v$; (iii) $f$ is a $6$-face with two special $2$-vertices $v,w$ such that $x,y$ are easy neighbors of $u$ in $G$ and $ux,uy$ intersect two edges incident with the $2$-vertices $v,w$.  Note that a poor $4$-face must be a $(d, 4^*, 2, 4^*)$-face with $d\ge 4$ by Lemma~\ref{lem7} and  a poor $6$-face must be a $(d, 4^*, 2, 4^*, 2, 4^*)$-face with $d\ge 2$ by Lemma~\ref{lem3}. A face is {\em semi-poor} if it is a $4$- or $5$-face with a $2$-vertex, or a $6^+$-face $f$ with one $8^+$-vertex $v$ and $(d(f)-1)$ $2$-vertices or $4^*$-vertices but not a poor $4$- or $6$-face. Note that a semi-poor $5$-face must be a $(7^+, 7^+, 4^*, 2, 4^*)$-face and a semi-poor $6^+$-face must be a $(8^+, 4^*, 2, 4^*, \ldots, 2, 4^*)$-face by Lemma~\ref{lem3}. See some examples in Figure 1.

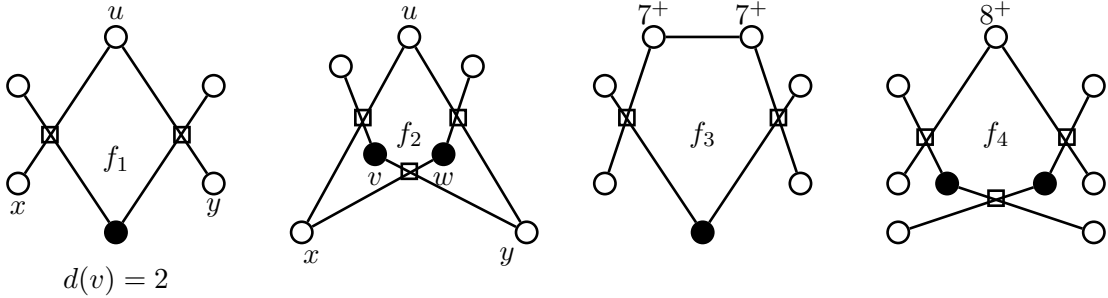
\begin{figure}[htb]\label{fig}
\centering
\begin{tikzpicture}[line width =1pt, scale=0.65, ]
\tikzstyle{point}=[circle, draw=black, inner sep=0.1cm]
\tikzstyle{square}=[rectangle, draw=black, inner sep=0.1cm]


\node[fill=black] (v11) at (1,3) [point] {};
\node[] (name)at(1,2) {$d(v)=2$};
\node (u11) at (1,7) [point] {};
\node (w11) at (-1,6) [point] {};
\node (w12) at (-1,4) [point] {};
\node (w13) at (3,6) [point] {};
\node (w14) at (3,4) [point] {};
\draw (u11)--(w12);
\draw (u11)--(w14);
\draw (v11)--(w11);
\draw (v11)--(w13);
\node[] (name)at(1,4.5) {$f_1$};
\node[] (name)at(-1,3.5) {$x$};
\node[] (name)at(3,3.5) {$y$};
\node[] (name)at(1,7.5) {$u$};
\node at (-0.35,5) [square] {};
\node at (2.35,5) [square] {};

\node[fill=black] (u21) at (6.3,4.6) [point] {};
\node[] (name)at(6.3,4.1) {$v$};
\node[fill=black]  (u22) at (7.7,4.6) [point] {};
\node[] (name)at(7.7,4.1) {$w$};
\node[]  (v21) at (4.8,3) [point] {};
\node[]  (v22) at (9.4,3) [point] {};
\node[]  (v23) at (7,7) [point] {};
\node[]  (w21) at (5.6,6.4) [point] {};
\node[]  (w22) at (8.3,6.4) [point] {};
\node[] (name)at(7,5) {$f_2$};
\node[] (name)at(7,7.5) {$u$};
\node[] (name)at(5,2.5) {$x$};
\node[] (name)at(9,2.5) {$y$};
\draw (u21)--(w21);
\draw (u21)--(v22);
\draw (u22)--(w22);
\draw (u22)--(v21);
\draw (v21)--(v23);
\draw (v22)--(v23);
\node at (6.05,5.35) [square] {};
\node at (8,5.35) [square] {};
\node at (7,4.25) [square] {};


\node[fill=black]  (u31) at (13,3) [point] {};
\node[]  (v31) at (12,7) [point] {};
\node[]  (v32) at (14,7) [point] {};
\node[]  (w31) at (11,6) [point] {};
\node[]  (w32) at (11,4) [point] {};
\node[]  (w33) at (15,6) [point] {};
\node[]  (w34) at (15,4) [point] {};
\node[] (name)at(13,5) {$f_3$};
\node[] (name)at(12,7.5) {$7^+$};
\node[] (name)at(14,7.5) {$7^+$};
\node at (11.45,5.35) [square] {};
\node at (14.55,5.35) [square] {};
\draw (u31)--(w31);
\draw (u31)--(w33);
\draw (v31)--(w32);
\draw (v32)--(w34);
\draw (v31)--(v32);


\node[fill=black]  (u41) at (18,4) [point] {};
\node[fill=black]  (u42) at (20,4) [point] {};
\node[]  (v41) at (19,7) [point] {};
\node[]  (w41) at (17,6) [point] {};
\node[]  (w42) at (17,4) [point] {};
\node[]  (w43) at (21,6) [point] {};
\node[]  (w44) at (21,4) [point] {};
\node[]  (w45) at (17,3) [point] {};
\node[]  (w46) at (21,3) [point] {};
\node[] (name)at(19,5) {$f_4$};
\node[] (name)at(19,7.5) {$8^+$};
\node at (17.55,4.95) [square] {};
\node at (20.45,4.95) [square] {};
\node at (19,3.7) [square] {};

\draw (u41)--(w41);
\draw (u41)--(w46);
\draw (u42)--(w45);
\draw (u42)--(w43);
\draw (v41)--(w42);
\draw (v41)--(w44);
\end{tikzpicture}
\caption{$f_1$ is a poor $4$-face ($x,y$ are easy neighbors of $u$), $f_2$ is a poor $6$-face ($x,y$ are easy neighbors of $u$), $f_3$ is a semi-poor $5$-face, and $f_4$ is a semi-poor $6$-face. All solid vertices are $2$-vertices and all squared vertices are $4^*$-vertices.}
\end{figure}

Below are our discharging rules:

\begin{itemize}
\item[(R1)] Each $8^+$-vertex gives $1$ to each incident poor $3^+$-face and $\frac{1}{2}$ to each incident semi-poor $4^+$-face or non-poor $3$-face.

\item[(R2)] Each $7$-vertex gives $\frac{1}{2}$ to each incident semi-poor $4$-face or semi-poor $5$-face with two $7$-vertices or $3$-face with a $4^*$-vertex. If $v$ is a non-special $7$-vertex on a $(7,7,10^+)$-face $f$, then $v$ gives $\frac{1}{2}$ to $f$.

\item[(R3)] A $5^+$-face $f$ gives $\frac{d(f)-4}{n_2(f)}$ to each incident $2$-vertex, where $n_2(f)$ is the number of $2$-vertices on $f$.

\item[(R4)] A $4^+$-face $f$ gives additional $\frac{m}{n'_2(f)}$ to each incident special $2$-vertex, where $n_2'(f)$ is the number of special $2$-vertices on $f$, and $m$ is the total charge that $f$ obtains from its incident vertices.
\end{itemize}

\begin{lemma}
Each face $f$ in $G^*$ has $\mu^*(f)\ge 0$.
\end{lemma}

\begin{proof}
Let $f$ be a face in $G^*$. By Lemma~\ref{lem5}, $d(f)\ge 3$.  If $d(f)=3$, then $f=[uvw]$ is either a $(4^*,7^+,7^+)$-face or a $(7^+,8^+,8^+)$-face or a $(7,7,10^+)$-face by Lemmas~\ref{lem3},~\ref{lem4} and~\ref{lem8}. In the first two cases, By (R1)(R2) $f$ gets $\frac{1}{2}$ from each of $v$ and $w$, so $\mu^*(f)\ge3-4+\frac{1}{2}\cdot2=0$. In the last case, $f$ gets $1$ from $w$ if both $u$ and $v$ are special by (R1),  and $f$ gets $\frac{1}{2}$ from $v$ and $w$ if at most one of $u,v$, say $u$, is special.  So $\mu^*(f)\ge 3-4+\max\{1,\frac{1}{2}\cdot 2\}=0$. If $d(f)\ge 4$, then by (R3) and (R4), $\mu^*(f)=0$.
\end{proof}

\begin{lemma}\label{5face}
Let $v$ be $2$-vertex and $u_1uvww_1$ be a segment of the boundary of $5^+$-face $f$. If none of $u_1$ and $w_1$ is a $2$-vertex, then $v$ gets at least $\frac{3}{2}$ from $f$ if $d(f)=5$ and at least $2$ from $f$ if $d(f)\ge 6$.
\end{lemma}

\begin{proof}
Assume that none of $u_1$ and $w_1$ is a $2$-vertex.  By Lemma~\ref{lem2}, $u_1$ and $w_1$ are $4^+$-vertices in $G$.
Let $d(f)=5$ first. Then $u_1w_1\in E(G)$ has no crossing. Thus $d(u_1),d(w_1)\ge 7$ by Lemma~\ref{lem3} and $f$ is semi-poor. So $f$ gets at least $\frac{1}{2}$ from $u_1$ and $w_1$ by (R1)(R2) and $v$ gets $5-4+\frac{1}{2}=\frac{3}{2}$ from $f$ by (R3) and (R4). Assume that $d(f)\ge 6$.  Then $f$ is incident to at most $\lceil \frac{d(f)-5}{2}\rceil$ $2$-vertices since each neighbor of a $2$-vertex is a $4^*$-vertex and two $4^*$-vertices cannot be adjacent in $G^*$ and the neighbors of $u_1,w_1$ on $f$ are not $2$-vertices. So $v$ gets at least $\frac{d(f)-4}{\lceil \frac{d(f)-5}{2}\rceil}\ge2$ from $f$ by (R3).
\end{proof}

\begin{lemma}
Each $2$-vertex $v$ in $G^*$ has $\mu^*(v)\ge 0$.
\end{lemma}

\begin{proof}
Let $N_H(v)=\{u,w\}$ and $f_1,f_2$ be the two faces incident to $v$. Then $v$ has initial charge $-2$. Since two $2$-vertices cannot be adjacent in $G^*$ by Lemma~\ref{lem3} and two $4^*$-vertices cannot be adjacent in $G^*$ by the 1-planar property of $G$, $f_i$ is incident to at most $\frac{d(f_i)-3}{2}$ $2$-vertices when $d(f_i)$ is odd and at most $\frac{d(f_i)}{2}$ $2$-vertices when $d(f_i)$ is even for each $i\in[2]$. If neither $f_1$ nor $f_2$ is a $4$-face or $6$-face with three $2$-vertices, then for each $i\in[2]$, $v$ gets at least $\frac{d(f)-4}{n_2(f)}\ge1$ from each of $f_1$ and $f_2$ by (R3), so $\mu^*(v)\ge-2+1+1=0$.

Let $f_1=[uu_1u_2\cdots u_{t-3}wv]$ and $f_2=[uvww_1w_2\cdots w_{k-3}]$, where $t=d(f_1),k=d(f_2)$. Then we have the following two remaining cases. 

{\bf Case 1:} $f_1$ is a $6$-face with three $2$-vertices. Then $d(v)=d(u_1)=d(u_3)=2$ and by Lemma~\ref{lem5} $f_2$ is a $4^+$-face.

Note that $u_2$ is a $4^*$-vertex. If $d(f_2)=4$, then we can modify the embedding to avoid the crossing vertex $u_2$ by switching the locations of $u_1$ and $u_3$, a contradiction. So we may assume that $f_2$ is a $5^+$-face.
By (R3), $v$ gets $\frac{d(f)-4}{n_2(f)}=\frac{2}{3}$ from $f_1$. Note that $u_1w_{k-3}, u_3w_1\in E(G)$. By Lemma~\ref{lem4},  $d(w_{k-3}),d(w_1)\ge7$. By Lemma~\ref{5face} $f$ gets at least  $\frac{3}{2}$ from $f_2$, so $\mu^*(v)\ge-2+\frac{2}{3}+\frac{3}{2}>0$.

{\bf Case 2:} $f_1$ is a $4$-face and $f_2$ is a $5^+$-face and incident to at most two $2$-vertices if $d(f_2)=6$.

If $d(f_2)=5$, then $f_2$ is a $(7^+,7^+,4^*,2,4^*)$-face by Lemma~\ref{lem3}.  By Lemma~\ref{5face}, $v$ gets at least  $\frac{3}{2}$ from $f_2$.  If $d(u_1)\ge7$, then $w_1$ gives $\frac{1}{2}$ to $f_1$ by (R1)(R2). By (R4) $f_1$ gives $\frac{1}{2}$ to $v$. So $\mu^*(v)\ge-2+\frac{1}{2}+\frac{3}{2}=0$. If $d(u_1)\le6$, then $d(w_1),d(w_2)\ge8$ by Lemma~\ref{lem4}. So $f_2$ gets $\frac{1}{2}\cdot2=1$ from $w_1, w_2$ and send $1+1=2$ to $v$ by (R3)(R4). So $\mu^*(v)\ge-2+2=0$.  Therefore we assume that $d(f_2)\ge 6$.

Let $d(f_2)=6$.  If $n_2(f_2)=1$, then $v$ gets $2$ from $f_2$, which implies that $\mu^*(v)\ge-2+2=0$. So we may assume that $n_2(f_2)=2$, by symmetry, say $d(w_1)=2$. Then $w_3$ is a $4^+$-vertex in $G$ by Lemma~\ref{lem2}. By Lemma~\ref{lem4}, $u_1$ and each vertex in $N_G(v)$ have degree at least $7$.  If $d(w_3)\le 7$, then $f_1$ is a poor $4$-face, thus $f_1$ gets $1$ from $u_1$ and sends it to $v$, and $\mu^*(v)\ge -2+1+1=0$ since $f_2$ sends at least $1$ to $v$ as well. So we assume that $d(w_3)\ge 8$. Note that $v$ is a special $2$-vertex. 
If $w_1$ is not a special $2$-vertex, then $f_2$ is a semi-poor $6$-face, thus by (R1) and (R4) $f_2$ gets $\frac{1}{2}$ from $w_3$ and sends $1+\frac{1}{2}$ to $v$, and $\mu^*(v)\ge -2+ \frac{3}{2}+\frac{1}{2}=0$ since $f_1$ sends $\frac{1}{2}$ to $v$ as well. Let $w_1$ be a special $2$-vertex as well. Then $f_2$ is a poor $6$-face to $w_3$, thus by (R1), $f_2$ gets $1$ from $w_3$ and sends $1+\frac{1}{2}$ to $v$ (and $w_1$), so $\mu^*(v)\ge -2+\frac{3}{2}+\frac{1}{2}=0$ since $f_1$ sends $\frac{1}{2}$ to $v$ as well.


Suppose that $d(f_2)\ge 7$. By Lemma~\ref{5face} we may assume that $w_1$ or $w_{k-3}$, say $w_1$, is a $2$-vertex. Then $d(u_1)\ge 7$ by Lemma~\ref{lem4}. If $d(w_{k-3})\le7$, then $f_1$ is a poor $4$-face and $d(u_1)\ge8$. So $f_1$ gets $1$ from $u_1$ and send it to $v$. Note that $f_2$ is incident to at most $\lfloor\frac{d(f_2)}{2}\rfloor$ $2$-vertices. So $v$ gets at least $1$ from $f_2$. So $\mu^*(v)\ge-2+1+1=0$.   Let  $d(w_{k-3})\ge8$. Then $f_2$ has at most $\lfloor\frac{d(f)-2}{2}\rfloor$ $2$-vertices. If $f_2$ is not a $8$-face with three $2$-vertices, then $v$ gets at least $\frac{d(f)-4}{\lfloor\frac{d(f)-2}{2}\rfloor}\ge\frac{3}{2}$ from $f_2$. So $\mu^*(v)\ge-2+\frac{1}{2}+\frac{3}{2}=0$. If $f_2$ is a $8$-face with three $2$-vertices, then $f_2$ is a semi-poor $8$-face. So $f_2$ gets $\frac{1}{2}$ from $w_{k-3}$ and sends additional $\frac{1}{6}$ to each incident $2$-vertex by (R1) and (R4). So $v$ gets at least $\frac{8-4}{3}+\frac{1}{6}=\frac{3}{2}$ from $f_2$. So $\mu^*(v)\ge-2+\frac{1}{2}+\frac{3}{2}=0$.
\end{proof}

\begin{lemma}
Each $3^+$-vertex $v$ in $G^*$ has $\mu^*(v)\ge 0$.
\end{lemma}

\begin{proof}
By Lemmas~\ref{lem1} and~\ref{lem2}, $d(v)\in \{4,6\}$ or $d(v)\ge 7$.  Since none of $4$-vertices and $6$-vertices are involved in the discharging process, $\mu^*(v)=d(v)-4\ge 0$ if $d(v)\in \{4,6\}$. So assume that $d(v)\ge 7$.

Suppose that $d(v)=7$. By (R2), $v$ gives either $\frac{1}{2}$ or $0$ to each incident face. Note that $v$ can afford to give six $\frac{1}{2}$, since the initial charge of $v$ is $3$. Therefore, we assume by contradiction that $v$ gives $\frac{1}{2}$ to each incident face by (R2).  By Lemma~\ref{lem4}, $v$ has at most one easy neighbor. If $v$ is incident to a semi-poor $4$-face or a semi-poor $5$-face with two $7$-vertices, say $f_2$, then $f_1$ or $f_3$ cannot be a $3$-face, for otherwise $v$ has at least two easy neighbors (namely, vertices with $2$-neighbors and/or $7$-vertex), a contradiction. So $f_1$ or $f_3$ is a  non-semi-poor $4^+$-face, who gets nothing from $v$ by (R2).  
So we assume that $v$ is incident seven $3$-faces each of which needs charge from $v$ by (R2). As $v$ has at most three $4^*$-neighbors, it is incident to a $3$-face without $4^*$-vertices, thus must be a $(7,7,10^+)$-face and $v$ is a  non-special $7$-vertex. Now that $v$ cannot have other easy neighbors in particular $7$-neighbors, all other incident $3$-faces must contain $4^*$-vertices, thus $v$ must be a special $7$-vertex, a contradiction.

Suppose that $d(v)\ge 8$. Let $m_p(v)$ be the number of poor $3^+$-faces incident to $v$. Note that each poor face contributes two easy neighbors to $v$ by the definition and any two poor faces of $v$ cannot be consecutive around $v$. So $m_p(v)\le \lfloor\frac{n_e(v)}{2}\rfloor\le\lfloor\frac{2d(v)-13}{2}\rfloor=d(v)-7$. By (R1) each $8^+$-vertex gives $1$ to each poor face and at most $\frac{1}{2}$ to each other incident face. If $m_p(v)\le d(v)-8$, then $\mu^*(v)\ge d(v)-4-1\cdot m_p(v)-\frac{1}{2}\cdot (d(v)-m_p(v))\ge \frac{1}{2}(d(v)-8-m_p(v))\ge 0.$ So we may assume that $m_p(v)=d(v)-7$. Then $\mu^*(v)\ge d(v)-4-1\cdot m_p(v)-\frac{1}{2}\cdot (d(v)-m_p(v))\ge \frac{1}{2}(d(v)-8-m_p(v))\ge -\frac{1}{2}.$ So we will show that we can save at least $\frac{1}{2}$ for $v$. If $v$ is incident to a poor $4^+$-face, say $f_2$, then $f_1$ or $f_3$, say $f_1$, cannot be a $3$-face since $v$ has at most $(2d(v)-13)$ easy neighbors. Note that $f_1$ is a $4^+$-face in $G^*$ with two non-adjacent $7^+$-vertices, thus cannot be semi-poor. So $f_1$ gets nothing from $v$, thus save $\frac{1}{2}$ for $v$. So we may assume that each incident poor face of $v$ is a poor $3$-face (that is, a $(10^+, 7,7)$-face with two special $7$-vertices). Recall that a $7$-vertex has at most one easy neighbor, so $d(v)$ must be even. Note that if $f_i$ is a poor $3$-face of $v$, then $v_{i-1}$ and $v_{i+2}$ must be $4^*$-vertices by the definition of poor $3$-face. Let $m_{4^*}(v)$ be the number of $4^*$-neighbors of $v$. We have $m_{4^*}(v)\ge m_p(v)+1$ and $m_{4^*}(v)+2m_p(v)\le d(v)$. So $m_p(v)\le\lfloor\frac{d(v)-1}{3}\rfloor$. Therefore, we can save at least  $\frac{1}{2}(\lfloor\frac{2d(v)-13}{2}\rfloor- \lfloor\frac{d(v)-1}{3}\rfloor)\ge1$ if $d(v)\ge12$. So we may assume that $d(v)=10$. Then $v$ is incident to three poor $3$-faces and has four $4^*$-neighbor, for otherwise, $\mu^*(v)\ge10-4-1\cdot2-\frac{1}{2}\cdot8=0$. So at least one incident face of $v$ is a $4^+$-face with two non-adjacent $8^+$-vertices, thus cannot be semi-poor and gets nothing from $v$. We also save $\frac{1}{2}$.
\end{proof}

\section{Appendix}
In this section, we assume that $G$ is a 1-planar graph that does not have an odd coloring of $c\ge7$ colors, but any graph with fewer vertices than $G$ has one. Consider $G^*$ is a plane drawing of a 1-planar graph such that each edge has at most one crossing point, and all crossing points are as few as possible.

\begin{lemma}(Claim 1 in ~\cite{CLS22+})
$G$ is $2$-edge-connected and thus $\delta(G)\ge 2$.
\end{lemma}

\begin{proof}
Suppose otherwise that $G$ is not 2-edge-connected. Let $e=uv$ be an edge in $G$ such that $G-e$ is disconnected. Let $G_1$ be the component of $G-e$ containing $u$ and $G_2$ be the component of $G-e$ containing $v$. By our assumption, $G_1$ has an odd $c$-coloring $\phi$ and $G_2$ has an odd $c$-coloring $\phi'$. We choose a color $a$ distinct from $\phi'_o(v)$,  exchange the two colors $a$ and $\phi(u)$ in $G_1$ when $\phi(u)\ne a$, and then choose a color $a'$ distinct from $a$ and $\phi_o(u)$, exchange the color $a'$ and $\phi'(v)$ in $G_2$ when $\phi'(v)\ne a'$. Now we guarantee that $u$ and $v$ have distinct colors and each of them has a color that occurs an odd number of times on its neighbors. So we get an odd $c$-coloring of $G$, a contradiction.
\end{proof}

\begin{lemma}(Claim 2 in ~\cite{CLS22+})
Every vertex of odd degree in $G$ has degree at least $\lfloor\frac{c+1}{2}\rfloor$.
\end{lemma}

\begin{proof}
Suppose otherwise that $G$ has a vertex $v$ of odd degree and $d(v)\le\lfloor\frac{c-1}{2}\rfloor$. By our assumption, $G-v$ has an odd $c$-coloring $\phi$. Then $v$ can be colored by avoiding at most $c-1$ colors, namely, the $\lfloor\frac{c-1}{2}\rfloor$ colors used for $v$'s neighbors and at most $\lfloor\frac{c-1}{2}\rfloor$ additional colors in regard to oddness concerning the neighbors of $v$. Since $v$ is of odd degree, there is at least one color appearing an odd number of times on $N_G(v)$. So we get an odd $c$-coloring of $G$ that extends $\phi$, a contradiction.
\end{proof}

\begin{lemma}(Claim 4 in ~\cite{CLS22+})
Every edge incident to a $\lfloor\frac{c-1}{2}\rfloor^-$-vertex in $G$ has a crossing.
\end{lemma}
\begin{proof}
Suppose otherwise that $v$ is a $\lfloor\frac{c-1}{2}\rfloor^-$-vertex such that  $e=uv$ has no crossing in $G$. Let $G'$ be the graph obtained from $G-v$ by adding an edge $uw$ for each $w\in N_G(v)\setminus N_G(u)$. Since edge $uv$ has no crossing, we can guarantee that each new edge $uw$ has a crossing if and only if $vw$ has a crossing. Thus $G'$ is 1-planar. By our assumption, $G'$ has an odd $c$-coloring $\phi$. Then $v$ can be colored by forbidding at most $c-1$ colors, namely, at most two colors in $\phi(w),\phi_o(w)$ for each $w\in N_G(v)$. Note that by our construction, $\phi(u)$ appears exactly once on $N_G(v)$, so we get an odd $c$-coloring of $G$, a contradiction.
\end{proof}

\begin{lemma} (Claim 6 in ~\cite{CLS22+})
The graph $G^*$ has no loop or $2$-face. Every $3$-face in $G^*$ is incident to either three $\lfloor\frac{c+1}{2}\rfloor^+$-vertices or two $\lfloor\frac{c+1}{2}\rfloor^+$-vertices and one $4^*$-vertex.
\end{lemma}

\begin{proof}
Clearly, $G^*$ cannot have any loops, since $G$ is loopless. Suppose that $G^*$ contains a $2$-face, with boundary $vw$. Since $G$ contains no parallel edges and two $4^*$-vertices are not adjacent in $G^*$, we may assume that $v\in V(G)$ and $w$ is a $4^*$-vertex. Note that $w$ arises in $G^*$ from a crossing of two edges $e,e'$ in $G$ that both incident to $v$. Then we can rotate $e$ and $e'$ to get an embedding of $G$ with fewer edge crossings. This contradicts our assumption, which prove the first statement.

Let $v$ be a $\lfloor\frac{c-1}{2}\rfloor^-$-vertex in $G$. By Lemma~\ref{lem4} every edge incident to $v$ has a crossing. So each neighbor of $v$ is a $4^*$-vertex. Recall that no two $4^*$-vertices are adjacent in $G^*$, so $v$ is not incident to any $3$-face in $G^*$. This proves the second statement.
\end{proof}

\begin{lemma} (Claim 7 in ~\cite{CLS22+})
Every $2$-vertex $v$ in $G^*$ is incident to a $5^+$-face and a $4^+$-face.
\end{lemma}

\begin{proof}
Let $v$ be a $2$-vertex in $G^*$ and let $u$ and $w$ be the neighbors of $v$ in $G^*$. By Lemma~\ref{lem4} both $u$ and $w$ are $4^*$-vertices. Recall that $G^*$ is 2-edge-connected. Let $f$ and $f'$ be the two distinct faces that contains the vertex $v$. So $u,v,w\in V_G^*(f)\cap V_G^*(f')$. Since $vw\in E(G^*)$, both $f$ and $f'$ are $4^+$-faces. Suppose that both $f$ and $f'$ are $4$-faces. Let $x,y$ be the remaining vertices of $f$ and $f'$, respectively. Now each vertex in $\{u,w\}$ is adjacent to both $x$ and $y$ in $G^*$. But this implies that $G$ has a multiple edge between $u$ and $w$, a contradiction.
\end{proof}

\begin{thebibliography}{9}
%

\bibitem{B84}
O.V. Borodin, Solution of the ringel problem on vertex-face coloring of planar graphs and coloring of 1-planar graphs,
Metody Diskret. Anal., 108(1984), 12--26.

\bibitem{B95}
O.V. Borodin, A new proof of the 6 color theorem, J. Graph Theory, 19(1995), 507--521.

\bibitem{BKRS01}
O. V. Borodin, A. V. Kostochka, A. Raspaud and E. Sopena, Acyclic colouring of 1-planar
graphs, Discrete Appl. Math., 114 (2001), 29--41.


\bibitem{BMWW08}
D. P. Bunde, K. Milans, D. B. West and H. Wu. Optimal strong parity edge-coloring of
complete graphs. Combinatorica, 28(2008),625--632.

\bibitem{C09}
P. Cheilaris, Conflict-Free Coloring, 2009, Thesis (Ph.D.)-City University of New York.

\bibitem{CKP13}
P. Cheilaris, B. Keszegh and D. P\'{a}lv\"{o}lgyi, Unique-maximum and conflict-free colouring
for hypergraphs and tree graphs, SIAM J. Discrete Math, 27 (2013) 1775–1787.


\bibitem{CCKP22}
Eun-Kyung Cho, Ilkyoo Choi, Hyemin Kwon, Boram Park, Odd coloring of sparse graphs and planar graphs, arxiv: 2202.11267.



\bibitem{C22+}
D.W. Cranston, Odd colorings of sparse graphs, Jan. 2022, arXiv:2201.01455v1.

\bibitem{CH13}
J. Czap, D. Hud\'{a}k, On drawings and decompositions of 1-planar graphs, Electron. J. Combin. 20(2013) P54.

\bibitem{CLS22+}
D.W. Cranston, M. Lafferty and Z. Song, A note on odd colorings of 1-planar graphs, Feb. 2022, arXiv:2202.02586v2.

\bibitem{DLM17}
Z. Dvo\v{r}\'{a}k, B. Lidick\'{y}, B. Mohar, 5-choosability of graphs with crossing far apart, J. Combin. Theory Ser. B, 123 (2017) 54--96.

\bibitem{ELRS03}
G. Even, Z. Lotker, D. Ron and S. Smorodinsky, Conflict-free colorings of simple geometric regions with applications to frequency assignment in cellular networks, SIAM J. Comput. 33 (2003) 94--136.

\bibitem{FG16}
I. Fabrici, F. G\"{o}ring, Unique-maximum coloring of plane graphs, Discuss. Math. Graph
Theory, 36(2016), 95--102.

\bibitem{FM07}
I. Fabrici, T. Madaras, The structure of 1-planar graphs, Discuss. Math. Graph
Theory, 307(2007), 854--865.

\bibitem{GST14}
R. Glebov, T. Szab\'{o} and G. Tardos, Conflict-free colouring of graphs, Combin. Probab.
Comput., 23 (2014) 434--448.

\bibitem{KKL12}
A. Kostochka, M. Kumbhat and T. Luczak, Conflict-free colourings of uniform hypergraphs
with few edges, Combin. Probab. Comput., 21(2012) 611--622.

\bibitem{KLM17}
S. Kobourov, G. Liotta and F. Montecchiani, An annotated bibliography on 1-planarity, Comp. Sci. Rev., 25 (2017) 49--67.

\bibitem{PP22+}
J. Petr, J. Portier, The odd chromatic number of a planar graph is at most 8, Feb. 2022, arXiv:2201.12381v2.

\bibitem{PS21+}
M. Petru\v{s}evski, R. \v{S}krekovski, Colorings with neighborhood parity condition, Dec. 2021, arXiv:2112.13710.


\bibitem{PT09}
J. Pach, G. Tardos, Conflict-free colourings of graphs and hypergraphs, Combin.
Probab. Comput., 18(2009) 819--834.

\bibitem{R65}
G. Ringel, Ein sechsfarbenproblem auf der kugel
Abh. Math. Semin.Univ. Hambg., 29(1965) 107--117.


\bibitem{YWW20}
W. Yang, W. Wang, and Y. Wang, An improved upper bound for the acyclic chromatic number of 1-planar graphs, Discrete Appl. Math., 283(2020), 275--291.

\bibitem{YWWL21}
W. Yang, Y. Wang, W. Wang and K. Lih, IC-planar graphs are 6-choosable,  SIAM J. . Discrete Math., 35(2021), 1729--1745.
\end{thebibliography}
\end{document}